\documentclass[a4paper, 11pt, abstracton, DIV=10]{scrartcl}

\usepackage[utf8]{inputenc}
\usepackage[T1]{fontenc}
\usepackage[english]{babel}
\usepackage{graphicx}
\usepackage{xcolor}
\usepackage{booktabs}
\usepackage{enumerate}
\usepackage{todonotes}
\usepackage{amsmath, amsfonts, amssymb,amsthm}
\usepackage[colorlinks=true, linkcolor=blue, filecolor=magenta, urlcolor=cyan, citecolor=gray]{hyperref}
\usepackage{cleveref}
\crefname{chapter}{Chapter}{Chapters}
\crefname{section}{Section}{Sections}
\crefname{subsection}{Section}{Sections}
\crefname{subsubsection}{Section}{Sections}
\crefname{figure}{Figure}{Figures}
\crefname{table}{Table}{Tables}
\numberwithin{equation}{section}
\theoremstyle{definition}

\crefname{question}{Question}{Questions}
\newtheorem{defi}{Definition}[section]
\crefname{def}{Definition}{Definitions}

\crefname{ex}{Example}{Examples}
\theoremstyle{plain}
\newtheorem{thm}[defi]{Theorem}
\crefname{thm}{Theorem}{Theorems}
\newtheorem{conj}{Conjecture}[section]
\crefname{conj}{Conjecture}{Conjectures}

\crefname{lemma}{Lemma}{Lemmas}
\newtheorem{cor}[defi]{Corollary}
\crefname{cor}{Corollary}{Corollaries}
\newtheorem{claim}[defi]{Claim}
\crefname{claim}{Claim}{Claims}

\crefname{prop}{Proposition}{Propositions}

\crefname{obs}{Observation}{Observations}
\newtheorem{prob}[conj]{Problem}
\crefname{prob}{Problem}{Problems}
\theoremstyle{remark}
\newtheorem{rmk}[defi]{Remark}
\crefname{rmk}{Remark}{Remarks}

\newcommand{\R}{\mathbb{R}}

\newcommand{\N}{\mathbb{N}}

\newcommand{\e}{\varepsilon}

\newcommand{\norm}[1]{\left\|#1 \right\|}

\newcommand{\diam}[1]{\mathrm{diam}(#1)}

\newcommand{\circr}[1]{\mathrm{cr}(#1)}

\author{Jan Corsten\thanks{Department of Mathematics, LSE, Email: \href{mailto:j.corsten@lse.ac.uk}{j.corsten@lse.ac.uk}, \href{mailto:n.frankl@lse.ac.uk}{n.frankl@lse.ac.uk}.} \and N\'ora Frankl\footnotemark[2] \footnote{Partially supported by the National Research, Development,
and Innovation Office, NKFIH Grant K119670.}}
\date{}
\title{A note on diameter-Ramsey sets\thanks{To appear in European J. Combin. \textbf{71} (2018), 51--54, DOI: \href{https://doi.org/10.1016/j.ejc.2018.02.036}{10.1016/j.ejc.2018.02.036}. \copyright~2018. This manuscript version is made available under the \href{http://creativecommons.org/licenses/by-nc-nd/4.0/}{CC-BY-NC-ND 4.0} license.}}

\begin{document}
\maketitle
\begin{abstract}
A finite set $A \subset \mathbb{R}^d$ is called \emph{diameter-Ramsey} if for every $r \in \N$, there exists some $n \in \N$ and a finite set $B \subset \mathbb{R}^n$ with $\diam{A}=\diam{B}$ such that whenever $B$ is coloured with $r$ colours, there is a monochromatic set $A' \subset B$ which is congruent to $A$.
We prove that sets of diameter $1$ with circumradius larger than $1/\sqrt{2}$ are not diameter-Ramsey. In particular, we obtain that triangles with an angle larger than $135^\circ$ are not diameter-Ramsey, improving a result of Frankl, Pach, Reiher and Rödl. Furthermore, we deduce that there are simplices which are almost regular but not diameter-Ramsey.
\end{abstract}

\section{Introduction}
In this note, we discuss questions related to \emph{Euclidean Ramsey theory}, a field introduced in \cite{EGMR73} by Erd\H os, Graham, Montgomery, Rothschild, Spencer and Straus. A finite set $A \subset \mathbb{R}^d$ is called \emph{Ramsey} if for every $r \in \N$, there exists some $n \in \N$ such that in every colouring of $\mathbb{R}^n$ with $r$ colours, there is a monochromatic set $A' \subset \R^n$ which is congruent to $A$. The problem of classifying which sets are Ramsey has been widely studied and is still open (see \cite{Handbook} for more details). 

The \emph{diameter} of a set $P \subset \R^d$ is defined by $\diam P := \sup \{ \norm{x-y} : x,y \in P \}$, where $\norm{\cdot}$ denotes the Euclidean norm. Recently,  Frankl, Pach, Reiher and Rödl \cite{frankl2017} introduced the following stronger property. 

\begin{defi} A finite set $A \subset \mathbb{R}^d$ is called \emph{diameter-Ramsey} if for every $r \in \N$, there exists some $n \in \N$ and a finite set $B \subset \mathbb{R}^n$ with $\diam{A}=\diam{B}$ such that whenever $B$ is coloured with $r$ colours, there is a monochromatic set $A' \subset B$ which is congruent to $A$.
\end{defi}

It follows from the definition that every diameter-Ramsey set is Ramsey. A set $A \subset \R^d$ is called spherical, if it lies on some $d$-dimensional sphere and the \emph{circumradius} of $A$, denoted by $ \circr A$, is the radius of the smallest sphere containing $A$.
(Note that if $A$ is spherical and is not contained in a proper subspace of $\mathbb{R}^d$, then there is a unique sphere that contains it.)
In \cite{EGMR73} it was proved that every Ramsey set must be spherical. Our main result states that every diameter-Ramsey set must also have a small circumradius.

\begin{thm}\label{thm:diamramsey}
If $ A \subset \mathbb{R}^d $ is a finite, spherical set with circumradius strictly larger than $ \diam A/\sqrt2 $, then $A$ is not diameter-Ramsey.
\end{thm}

 Frankl, Pach, Reiher and Rödl \cite[Theorems 3 and 4]{frankl2017} proved that acute and right-angled triangles are diameter-Ramsey, while triangles having an angle larger than $150^\circ$ are not. \cref{thm:diamramsey} implies the following improvement.

\begin{cor}\label{cor:triangle}
Triangles with an angle larger than $135^\circ$ are not diameter-Ramsey.
\end{cor}

Let us call a $d$-simplex $ A = \{p_1, \ldots, p_{d+1} \}$ \emph{$\e$-almost regular} if \[ \frac{1}{\binom{d+1}{2}} \sum_{1 \leq i < j \leq d+1} \diam A^2 - \norm{p_i -p_j}^2 \leq \e \cdot \diam A^2.\] In \cite[Theorem 6, Lemma 4.9]{frankl2017} it was further proved that $ \e $-almost regular simplices are diameter-Ramsey for every $ \e \leq 1/\binom{d+1}{2} $. This is a rather small class of simplices since $ 1/\binom{d+1}{2} $ tends to zero, but another corollary of \cref{thm:diamramsey} shows that one cannot hope for much more. 

\begin{cor}\label{cor:simplex}
For every $ d \in \N $ and every $ \e > \sqrt{d}/\binom{d+1}{2} $, there is an $ \e $-almost regular $d$-simplex which is not diameter-Ramsey.   
\end{cor}

For $d \in \N$ and $ r \geq 0 $, we denote the closed $d$-dimensional ball of radius $ r $ centred at the origin by $ B_d(r) $.
We will deduce \cref{thm:diamramsey} from the following result.

\begin{thm}\label{thm:main}
For every finite, spherical set $ A \subset \R^d $ and every positive number $ r < \circr A $, there is some $ k = k(A,r) \in \N$ such that the following holds. For every $ D \in \N $, there is a colouring of $ B_D(r) $ with $ k $ colours and with no monochromatic, congruent copy of $A$.
\end{thm}

A result of Matoušek and Rödl \cite{MatRodl} shows that the conclusion of \cref{thm:main} does not hold whenever $r > \circr A$. We do not know what happens when $r = \circr A$.

\begin{rmk}
After completing this work, we have learnt that Theorem~\ref{thm:diamramsey} has independently been proved by  Frankl, Pach, Reiher and Rödl, with a similar proof (J\'anos Pach, private communication).
\end{rmk}

\section{Proofs}
\subsection{Proof of \texorpdfstring{\cref{thm:main}}{the main theorem}}
Fix some finite, spherical $A \subset \R^d$ and some positive number $r < \circr A$. The following claim is the key step of the proof.

\begin{claim}\label{claim:key}
There exists a constant $ c = c(A,r) > 0 $ such that for every $D \in \N$ and for every congruent copy $A'$ of $A$ in $B_{D}(r)$ we have $\max_{x,y \in A'} \left( \norm x - \norm y \right) \geq c$. 
\end{claim}

\begin{proof} First observe that it is sufficient to prove the claim for $D = d+1$. For $D < d+1 $, this follows immediately from $B_{D}(r) \subset B_{d+1}(r)$, and for $D>d+1$ we can consider the at most $(d+1)$-dimensional subspace spanned by the vertices of $A'$ and the origin.

Let $E=\{e:A\to B_D(r)\} \subset B_D(r)^{|A|}$ be the set of all embeddings of $A$ to $B_D$. It is easy to see that, if $e_1, e_2, \ldots \in E$ and the pointwise limit $e := \lim_n e_n$ exists, then $ e \in E$. Therefore, $E$ is a closed subset of a compact metric space and hence $E$ is compact as well. Define $f:E\to \mathbb{R}$ by \[f(e) := \max_{x,y \in e(A)} \left(\norm x - \norm y \right). \]

Clearly, $f(e)\geq 0$ for every $ e \in E $, and $f(e) = 0$, if and only if $e(A)$ lies on a sphere around the origin. But since $\circr{e(A)} > r$ for every embedding $ e \in E$, this is not the case, and hence $f(e)>0$ for all $e \in E$. Finally, since $f$ is continuous, there is a constant $c >0$ such that $f(e) \geq c$ for all $e \in E$.
\end{proof}

\medskip

  Let $ k = \lfloor r/c \rfloor + 1 $ now, and fix some $ D \in \N $. We will colour points in $ B_D(r) $ by their distance to the origin: Define $ \chi : B_D(r) \to \{0, \ldots, k-1 \} $ by $ \chi(x) = \lfloor 1/c \cdot \norm{x} \rfloor $, and let $A' \subset B_D(r)$ be a congruent copy of $A$. It follows immediately from \cref{claim:key} that there are $x,y \in A'$ with $\norm{x}-\norm{y} \geq c$, and hence $\chi (x) \not = \chi (y)$. This finishes the proof of \texorpdfstring{\cref{thm:main}}{the main theorem}.

\subsection{Implications of \texorpdfstring{\cref{thm:main}}{the main theorem}}
In this section we will deduce \cref{thm:diamramsey} and then \cref{cor:triangle,cor:simplex}. In order to do so we will use the following classical result.

\begin{thm}[Jung's inequality, \cite{Jung1901}]\label{thm:jung}
Every bounded set $ A \subset \R^d $ can be covered by a closed ball of radius $\sqrt{d/(2d+2)} \cdot \diam A$.
\end{thm}

In particular, every finite set $B \subset \R^n$ can be covered by a ball of radius $\diam{B} /\sqrt{2}$, and hence \cref{thm:diamramsey} follows immediately from \cref{thm:main,thm:jung}.

Furthermore, if $T$ is a triangle with an angle $\alpha>135^\circ$ and diameter $a$, it is folklore that the circumradius of $T$ is $a/(2\sin \alpha)>a/\sqrt{2}$. Thus, we obtain \cref{cor:triangle} as a corollary of \cref{thm:diamramsey}.

To prove \cref{cor:simplex}, we show that we can move one vertex of the regular $d$-simplex by just a little bit to obtain a simplex of circumradius strictly larger than $ 1/\sqrt{2} $. We will use the elementary geometric fact that the circumradius of a $d$-dimensional unit simplex is $\sqrt{d/(2d+2)}$.

\begin{proof}[Proof of \cref{cor:simplex}]
Let $\delta >0 $, $r^2 = 1/2 + \delta$, $a^2 = 1/(2d) +\delta $, and define $ H_a := \{ x \in \R^d : x_d = a \}$. Then $ B:=B_d(r) \cap H_a $ is a $(d-1)$-dimensional ball of radius $ \sqrt{r^2 - a^2} = \sqrt{(d-1)/(2d)} $, and hence there is a $(d-1)$-dimensional unit simplex $A'=\{p_1, \ldots, p_d\}$ contained in the boundary of $B$. Finally, let $ A = A' \cup \{p_{d+1}\} $, where $ p_{d+1} = (0, \ldots, 0, r) $.

By construction, we have $ \circr A > 1/\sqrt{2} $, $ \norm{p_i-p_j}^2 = 1 $ for all $ 1 \leq i < j \leq d $, and $ \norm{p_i - p_{d+1}}^2 = 1 - 1/\sqrt{d} + O(\delta) $ for all $ 1 \leq i \leq d$. Hence the theorem follows from \cref{thm:diamramsey} after choosing $ \delta > 0 $ small enough.
\end{proof}

\section{Remarks}

In \cite{frankl2017} it was asked whether there exists an obtuse triangle which is diameter-Ramsey. Although we could not answer this question, we think the answer is no. More generally, we think the following statement is true.

\begin{conj}\label{conj:circumcenter}
A simplex is diameter-Ramsey if and only if its circumcentre is contained in its convex hull.
\end{conj}

Furthermore, it would be interesting to close the gap between \cref{cor:simplex} and the related result from \cite{frankl2017}.

\begin{prob}
For every $d \in \N$, determine the largest $\e = \e(d) >0$, such that every $\e$-almost-regular simplex is diameter-Ramsey.
\end{prob}

Note that, provided \cref{conj:circumcenter} is true, a similar construction as in the proof of \cref{cor:simplex} shows that the result in \cite{frankl2017} is best possible, i.e.\ $\e(d) = 1/\binom{d}{2}$.

\section*{Acknowledgement}
The authors would like to thank their supervisors Peter Allen, Julia B\"ottcher, Jozef Skokan and Konrad Swanepoel for their helpful comments on this note.

\bibliographystyle{amsplain}
\bibliography{./input/bibinits}
\end{document}